\documentclass[a4paper, 11pt]{amsart}
\usepackage{amsmath,mathrsfs}
\usepackage{amsfonts}
\usepackage{amssymb}
\usepackage{graphicx}
\usepackage[square,numbers]{natbib}
\newtheorem{theorem}{Theorem}

\newtheorem{lemma}[theorem]{Lemma}
\newtheorem{proposition}[theorem]{Proposition}

\theoremstyle{definition}
\newtheorem{definition}[theorem]{Definition}

\theoremstyle{remark}
\newtheorem{remark}[theorem]{Remark}

\numberwithin{theorem}{section}
\numberwithin{equation}{section}

\def\XXint#1#2#3{{\setbox0=\hbox{$#1{#2#3}{\int}$}
     \vcenter{\hbox{$#2#3$}}\kern-.5\wd0}}

\newcommand{\dd}{\; \mathrm{d}}

\begin{document}
\title[Lusin Approximation and Horizontal Curves in Carnot Groups]{Lusin Approximation and Horizontal Curves in Carnot Groups}
\author{Gareth Speight}
\address{Department of Mathematical Sciences, University of Cincinnati, Cincinnati, OH 45221, United States}
\email{Gareth.Speight@uc.edu}

\begin{abstract}
We show that, given an absolutely continuous horizontal curve $\gamma$ in the Heisenberg group, there is a $C^{1}$ horizontal curve $\Gamma$ such that $\Gamma=\gamma$ and $\Gamma'=\gamma'$ outside a set of small measure. Conversely, we construct an absolutely continuous horizontal curve in the Engel group with no $C^{1}$ horizontal approximation.
\end{abstract}
\maketitle

\section{Introduction}

A curve $\gamma\colon [a,b] \to \mathbb{R}^n$ is absolutely continuous if it is differentiable Lebesgue almost everywhere and any displacement $\gamma(t_{2})-\gamma(t_{1})$ is given by integrating the derivative $\gamma'$ from $t_{1}$ to $t_{2}$ (Definition \ref{absolutecontinuity}). This connection between position and derivative makes absolute continuity an important property in analysis and geometry.

It is well known that absolutely continuous curves in $\mathbb{R}^{n}$ admit a Lusin type approximation by $C^{1}$ curves: given any absolutely continuous curve in $\mathbb{R}^{n}$, one can find a $C^{1}$ curve with the same position and derivative except for a set of small measure (Theorem \ref{classicallusin}). In this article we investigate what happens if curves cannot move freely but instead are constrained to move in a smaller, but still rich, family of directions. 

A sub-Riemannian manifold consists of a smooth manifold together with a completely non-integrable subbundle of the tangent bundle (called the horizontal subbundle) and a smoothly varying family of inner products on the horizontal subbundle. A vector field is called horizontal if it is everywhere tangent to the horizontal subbundle. Non-integrability of the horizontal subbundle means that any tangent vector is a linear combination of vectors which arise as (possibly iterated) Lie brackets of horizontal vector fields. This non-integrability implies that distinct points in a connected sub-Riemannian manifold can be connected by absolutely continuous horizontal curves, which have tangent vectors in the horizontal subbundle (Chow-Rashevskii Theorem \cite{Cho39, Ras38}). One can then define the Carnot-Caratheodory distance between points by considering lengths of such horizontal curves. 

A much investigated class of sub-Riemannian manifolds is that of Carnot groups (Definition \ref{Carnot}). Carnot groups have the structure of a Lie group (so possess a family of translations) whose Lie algebra admits a stratification (which gives rise to a family of dilations). Such spaces have some similarities with Euclidean spaces but also many striking differences. For example, each Carnot group carries a natural Haar measure and one can prove a meaningful version of Rademacher's theorem for Lipschitz maps between Carnot groups (Pansu's Theorem \cite{Mag01, Pan89}). On the other hand, any Carnot group (except for Euclidean spaces) contains no subset of positive measure which is bi-Lipschitz equivalent to a subset of a Euclidean space \cite{LeD, Sem96}.

Since only horizontal curves are relevant, it is not appropriate to apply a classical Lusin type $C^{1}$ approximation to horizontal curves in the sub-Riemannian setting. We ask if a meaningful generalization of Theorem \ref{classicallusin} holds in Carnot groups: does any absolutely continuous horizontal curve in a Carnot group have a Lusin approximation by a $C^{1}$ horizontal curve? 

We show that the answer to this question is yes in Heisenberg groups (Definition \ref{Heisenberg} and Theorem \ref{lusinheisenberg}), the simplest and most studied type of non-Euclidean Carnot group. On the contrary, we show the answer to the question is no in the first Engel group (Definition \ref{Engel} and Theorem \ref{nolusin}); hence Theorem \ref{classicallusin} does not allow a generalization to horizontal curves in general Carnot groups.

By using exponential coordinates, every Carnot group can be viewed as $\mathbb{R}^{n}$ equipped with a particular family of vector fields (of which some are the distinguished horizontal directions); the group operation is then calculated from these vector fields using the Baker-Campbell-Hausdorff formula. Throughout this article we use such concrete representations and elementary techniques as much as possible. We now explain the main ideas in the proofs of Theorem \ref{lusinheisenberg} and Theorem \ref{nolusin}.

The proof of Theorem \ref{lusinheisenberg} in the Heisenberg group, represented in $\mathbb{R}^{2n+1}$, combines measure theoretic ideas from the proof of the classical Theorem \ref{classicallusin} (e.g. see \cite[Lemma 3.4.8]{AT04}) and geometric ideas for the Heisenberg group. One uses measure theoretic techniques to restrict to a compact set $K$ of large measure such that $\gamma'|_{K}$ is uniformly continuous and each point of $K$ is a Lebesgue point of $\gamma'$. The complement of such a compact set is a union of intervals $(a,b)$. Given such an interval, we find a horizontal $C^{1}$ curve $\Gamma\colon [a,b]\to\mathbb{R}^{2n+1}$, with $\Gamma(t)=\gamma(t)$ and $\Gamma'(t)=\gamma'(t)$ for $t=a$ and $t=b$, such that $\Gamma'(t)$ is close to $\gamma'(a)$ for all $t\in (a,b)$. We do this by using the fact that horizontal curves in the Heisenberg group are lifts of arbitrary curves in $\mathbb{R}^{2n}$ (Lemma \ref{lift}). The height of such a lift is easily prescribed since it is given by an integral which can be interpreted as a signed area in $\mathbb{R}^{2n}$ (Lemma \ref{area}).

The proof of Theorem \ref{nolusin} in the Engel group uses carefully the form of the horizontal vector fields \eqref{engelvectors}. The horizontal vector fields $X_{1}$ and $X_{2}$ have the property that if a horizontal curve $\gamma$ has a direction close to $X_{2}$, then the component $\gamma_{4}$ will always be increasing. However, for example by following the direction $-X_{2}$, the component $\gamma_{4}$ definitely can decrease for a general horizontal curve. We exploit this observation to construct an absolutely continuous horizontal curve and a dense set $[0,1]\setminus S$ of large measure but empty interior for which the following holds: $\gamma$ follows the vector field $X_{2}$ on $[0,1]\setminus S$, but the component $\gamma_{4}$ is strictly decreasing on $[0,1]\setminus S$ (Proposition \ref{badcurveproperties}). We then show that, given a $C^{1}$ horizontal curve $\Gamma$, the set $\{t\in [0,1]\setminus S:\Gamma(t)=\gamma(t)\}$ has no Lebesgue density points and hence has measure zero.

To conclude this introduction we give those definitions which will be important throughout the paper. Note that the phrases `almost every point' or `almost everywhere' will mean for all points except those in a subset of $\mathbb{R}$ with Lebesgue measure zero.

We now give basic information about absolutely continuous functions and Lusin type approximations, as can be found in \cite{EG91, Rud86}.

\begin{definition}\label{absolutecontinuity}
A curve $\gamma\colon [a,b]\to \mathbb{R}^{n}$ is absolutely continuous if it is differentiable almost everywhere and
\[\gamma(t_{2})=\gamma(t_{1})+\int_{t_{1}}^{t_{2}} \gamma'(t) \dd t\]
whenever $t_{1}, t_{2}\in [a,b]$. A curve in a smooth manifold is absolutely continuous if (locally) its image under any chart map is absolutely continuous.
\end{definition}

An equivalent definition of absolute continuity in $\mathbb{R}^n$ is given by requiring that for every $\varepsilon>0$ there exists $\delta>0$ such that the following holds: whenever $(a_{i},b_{i})$ is a finite pairwise disjoint sequence of subintervals of $[a,b]$ satisfying
\[\sum_{i} (b_{i}-a_{i})<\delta,\]
we have
\[\sum_{i} |\gamma(b_{i})-\gamma(a_{i})|<\varepsilon.\]
From this definition it is immediate that every Lipschitz curve is absolutely continuous. The Lusin type $C^{1}$ approximation theorem for absolutely continuous mappings is as follows.

\begin{theorem}\label{classicallusin}
Suppose $\gamma\colon [a,b]\to \mathbb{R}^{n}$ is absolutely continuous and $\varepsilon>0$. Then there exists a $C^1$ curve $\Gamma\colon [a,b]\to \mathbb{R}^{n}$ such that:
\[\mathcal{L}^1 (\{t\in [a,b]:\Gamma(t)\neq \gamma(t) \mbox{ or }\Gamma'(t)\neq \gamma'(t)\})<\varepsilon.\]
\end{theorem}

We now give some basic facts about Carnot groups; more information can be found in \cite{ABB, B94, BLU07, LeD, Mon02, Vit14}. A Lie group is a smooth manifold with the structure of a group, for which the group operations are smooth. The Lie algebra of a Lie group consists of left invariant vector fields (which can be identified with the tangent space at the identity). The Lie bracket of smooth vector fields $X, Y$ on a smooth manifold is another smooth vector field $[X,Y]$, defined on smooth functions by $[X,Y](f)=X(Y(f))-Y(X(f))$. The abstract definition of Carnot group is as follows, though later we will work with concrete representations in $\mathbb{R}^{n}$.

\begin{definition}\label{Carnot}
A Carnot group $\mathbb{G}$ of step $s$ is a simply connected Lie group whose Lie algebra $\mathcal{G}$ admits a stratification of the form
\[\mathcal{G}=V_{1}\oplus V_{2}\oplus \ldots \oplus V_{s}\]
such that $V_{i}=[V_{1},V_{i-1}]$ for any $i=2, \ldots, s$, and $[V_{1},V_{s}]=0$.
We refer to $V_{1}$ as the horizontal layer and say that an absolutely continuous curve in $\mathbb{G}$ is horizontal if its derivative belongs to $V_{1}$ at almost every point.
\end{definition}

Using the stratification of the Lie algebra, one can define a family of dilations on any Carnot group. We do not use them in this article but they have an important role in the geometry of Carnot groups. Since the Lie group composition is smooth, left translations $x\mapsto p\cdot x$ map absolutely continuous ($C^{1}$) curves to absolutely continuous ($C^{1}$) curves. Further, since the Lie algebra consists of left invariant vector fields, such translations map horizontal vectors (curves) to horizontal vectors (curves).

Recall that the exponential map of a Lie group is a mapping $\exp\colon \mathcal{G}\to \mathbb{G}$; it is defined by $\exp(X)=\gamma(1)$, where $\gamma\colon (\mathbb{R},+)\to (\mathbb{G},\cdot)$ is the unique group homomorphism such that $\gamma'(0)=X$. If $X$ is interpreted as a left invariant vector field, such a curve $\gamma$ is equivalently defined as the solution of the Cauchy problem $\gamma(0)=0, \gamma'(t)=X(\gamma(t))$.

For a Carnot group the exponential map is a diffeomorphism. This allows one to identify any Carnot group with its Lie algebra:
\[X\in \mathcal{G} \longleftrightarrow \exp(X)\in \mathbb{G}.\]
Adopting such an identification, the group law can be calculated from the Lie brackets using the Baker-Campbell-Hausdorff formula. We need this formula only for groups of step at most 3, for which it takes the form:
\[\exp(X)\cdot \exp(Y)=\exp \left( X+Y+\frac{1}{2}[X,Y]+\frac{1}{12}([X,[X,Y]]+[Y,[Y,X]])\right).\]
The Lie algebra of a Carnot group is a finite dimensional vector space. Hence, by choosing a basis, it can be identified with $\mathbb{R}^{n}$. Thus every Carnot group can be represented as $\mathbb{R}^n$ equipped with a suitable family of vector fields, from which the group operation is calculated using the Baker-Campbell-Hausdorff formula. We will represent both the Heisenberg group (Definition \ref{Heisenberg}) and the Engel group (Definition \ref{Engel}) in this way. 

\begin{remark}
We prove a positive result in the Heisenberg group (a Carnot group of step 2) and a negative result in the Engel group (a Carnot group of step 3). Hence our analysis of $C^{1}$ approximations leaves open the case of general step 2 Carnot groups; such groups are rather special because their group operation and vector fields can be written explicitly \cite{BLU07}. With a more algebraic argument, it may be possible to determine whether Theorem \ref{lusinheisenberg} holds for general step 2 Carnot groups.

Further, it may be possible to strengthen Theorem \ref{nolusin} by constructing an absolutely continuous horizontal curve in the Engel group which meets every $C^{1}$ horizontal curve in a set of measure zero (rather than a set of small measure). However, such a construction is likely to be more complicated and was not the main motivation behind this article.
\end{remark}

\noindent \textbf{Acknowledgement.} This work was carried out with the support of the grant ERC ADG GeMeThNES. The author thanks Luigi Ambrosio for asking the question to him and for much guidance during his time in Pisa. 

The author also thanks Francesco Serra Cassano for initially raising the problem and Andrea Pinamonti for useful conversations on related topics. 

The author is grateful to kind referees who read this article very carefully and suggested numerous corrections and improvements to the exposition.

\section{Lusin Approximation in the Heisenberg Group}

We now give the abstract definition of the Heisenberg group, followed by a representation in coordinates \cite{BLU07, CDPT07}.

\begin{definition}\label{Heisenberg}
The Heisenberg group $\mathbb{H}^{n}$ is the Carnot group whose Lie algebra $\mathcal{H}$ satisfies the decomposition
\[ \mathcal{H}=V_{1}\oplus V_{2},\]
where $V_{1}=\mathrm{span}\{X_{i},Y_{i}:1\leq i\leq n\}$, $V_{2}=\mathrm{span}\{T\}$, and the only non-trivial Lie brackets are
\[[X_{i},Y_{i}]=-4T, \qquad 1\leq i\leq n.\]
\end{definition}

The Heisenberg group $\mathbb{H}^{n}$ can be represented in coordinates by $\mathbb{R}^{2n+1}$, whose points we denote by $(x,y,t)$ with $x,y\in \mathbb{R}^{n}$ and $t\in \mathbb{R}$. The group law is given by:
\[(x,y,t)\cdot (x',y',t')=\left(x+x',y+y',t+t'+2\sum_{i=1}^{n}(y_{i}x_{i}'-x_{i}y_{i}')\right).\]
The identity is $(0,0,0)$ and inverses are given by $(x,y,t)^{-1}=(-x,-y,-t)$. The vector fields are:
\begin{equation}\label{heisenbergvectors}
X_{i}=\partial_{x_{i}}+2y_{i}\partial_{t}, \quad Y_{i} = \partial_{y_{i}}-2x_{i}\partial_{t}, \quad 1\leq i\leq n, \quad T=\partial_{t}.
\end{equation}
Here $\partial_{x_{i}}, \partial_{y_{i}}$ and $\partial_{t}$ denote the coordinate vectors in $\mathbb{R}^{2n+1}$, which may be interpreted as operators on differentiable functions.

We recall that an absolutely continuous curve $\gamma$ in the Heisenberg group is horizontal if at almost every point $t$ the derivative $\gamma'(t)$ is a linear combination of the vectors $X_{i}(\gamma(t)), Y_{i}(\gamma(t)), 1\leq i\leq n$. A vector in $\mathbb{R}^{2n+1}$ is horizontal at $\zeta \in \mathbb{R}^{2n+1}$ if it is a linear combination of the vectors $X_{i}(\zeta), Y_{i}(\zeta), 1\leq i\leq n$. In this section we prove the following result.

\begin{theorem}\label{lusinheisenberg}
Suppose $\gamma\colon [0,1]\to \mathbb{R}^{2n+1}$ is an absolutely continuous horizontal curve in the Heisenberg group. Then, for any $\varepsilon>0$, there exists a $C^{1}$ horizontal curve $\Gamma\colon [0,1]\to \mathbb{R}^{2n+1}$ such that:
\[\mathcal{L}^{1} (\{t\in [0,1]:\Gamma(t)\neq \gamma(t)\mbox{ or }\Gamma'(t)\neq \gamma'(t)\})<\varepsilon.\]
\end{theorem}

Before proving Theorem \ref{lusinheisenberg} we give some well known statements about horizontal curves in the Heisenberg group. The next lemma states that horizontal curves are given by suitable lifts of curves in $\mathbb{R}^{2n}$. It follows immediately from the form of the vector fields \eqref{heisenbergvectors}.

\begin{lemma}\label{lift}
An absolutely continuous curve $\gamma\colon [a,b]\to \mathbb{R}^{2n+1}$ is a horizontal curve in the Heisenberg group if and only if for $t\in [a,b]$:
\[\gamma_{2n+1}(t)=\gamma_{2n+1}(a)+2\sum_{i=1}^{n}\int_{a}^{t} (\gamma_{i}'\gamma_{n+i}-\gamma_{n+i}'\gamma_{i}).\]
\end{lemma}

\begin{proof}
By definition, $\gamma$ is horizontal if and only if there exist functions $A_{i}, B_{i}\colon [a,b] \to \mathbb{R}$, $1\leq i\leq n$, such that for almost every $t\in [a,b]$:
\[\gamma'(t)=\sum_{i=1}^n (A_{i}(t)X_{i}(\gamma(t))+B_{i}(t)Y_{i}(\gamma(t))).\]
Using $\eqref{heisenbergvectors}$, the right hand side of this expression is exactly
\[(A_{1}(t), \ldots, A_{n}(t), B_{1}(t), \ldots, B_{n}(t), \sum_{i=1}^{n} (2A_{i}(t)\gamma_{n+i}(t)-2B_{i}(t)\gamma_{i}(t))).\]
By examining the initial coordinates we see $\gamma_{i}'(t)=A_{i}(t)$ and $\gamma_{n+i}'(t)=B_{i}(t)$ for $1\leq i\leq n$. Hence we deduce:
\[\gamma_{2n+1}'(t)= 2\sum_{i=1}^{n} (\gamma_{i}'(t)\gamma_{n+i}(t)- \gamma_{n+i}'(t)\gamma_{i}(t))\]
for almost every $t\in [a,b]$. The claim then follows by integration.
\end{proof}

The integrals above have a geometric interpretation; if the curve starts at $0$ and is smooth enough to apply Stokes' theorem, then each one is given by a signed area in $\mathbb{R}^2$.

\begin{lemma}\label{area}
Suppose $\sigma\colon [a,b]\to \mathbb{R}^{2}$ is a smooth curve with $\sigma(a)=0$. Let $[0,\sigma(b)]$ be the straight line from $0$ to $\sigma(b)$ and let $A_{\sigma}$ denote the signed area of the region enclosed by $\sigma$ and $[0,\sigma(b)]$. Then:
\[A_{\sigma}=\frac{1}{2}\int_{a}^{b}(\sigma_{1}\sigma_{2}'-\sigma_{2}\sigma_{1}').\]
\end{lemma}

\begin{proof}
Define a differential 1-form by $\alpha=\frac{1}{2}(x\mathrm{d} y-y\mathrm{d} x)$ and notice that $\mathrm{d}\alpha=\mathrm{d} x\wedge \mathrm{d} y$ is the area form. Any straight line through the origin, in particular $[0,\sigma(b)]$, is contained in the kernel of $\alpha$. Since $\sigma$ followed by $[0,\sigma(b)]$ gives a closed curve we may apply Stokes' theorem:
\[\frac{1}{2}\int_{a}^{b}(\sigma_{1}\sigma_{2}'-\sigma_{2}\sigma_{1}')=\frac{1}{2}\int_{\sigma} \alpha =\frac{1}{2}\int_{\sigma\cup [0,\sigma(b)]} \alpha=\frac{1}{2}\int_{D_{\sigma}} \mathrm{d} \alpha=A_{\sigma}.\]
\end{proof}

Fix an absolutely continuous horizontal curve $\gamma\colon [0,1]\to \mathbb{R}^{2n+1}$ in the Heisenberg group. Then $\gamma'(t_{0})$ exists and is a horizontal vector for almost every $t_{0} \in [0,1]$. Since $\gamma$ is absolutely continuous we know that $\gamma'$ is integrable, hence almost every $t_{0}\in (0,1)$ is a Lebesgue point of $\gamma'$; in particular:
\[\lim_{r\downarrow 0} \frac{1}{r}\int_{t_{0}}^{t_{0}+r} |\gamma'-\gamma'(t_{0})| =0.\]
Now let $\varepsilon > 0$; using standard measure theoretic arguments and Lusin's theorem, we find a compact set $K \subset (0,1)$ such that $\mathcal{L}^1 ([0,1]\setminus K)<\varepsilon$ and the following statements hold:
\begin{itemize}
\item $\gamma'(t_{0})$ exists and is a horizontal vector at every $t_{0} \in K$.
\item $\gamma'|_{K}$ is uniformly continuous.
\item Each $t_{0}\in K$ is a Lebesgue point of $\gamma'$ with averages converging uniformly. More precisely: for every $\eta>0$ there exists $\delta(\eta)>0$ such that $t_{0}\in K$ and $0<r<\delta(\eta)$ implies
\[\int_{t_{0}}^{t_{0}+r} |\gamma'-\gamma'(t_{0})| \leq \eta r.\]
\end{itemize}

Let $m=\min K$ and $M=\max K$. To prove Theorem \ref{lusinheisenberg} it suffices to find a $C^1$ horizontal curve $\Gamma \colon [m,M]\to \mathbb{R}^{2n+1}$ such that $\Gamma=\gamma$ and $\Gamma'=\gamma'$ on $K$. After this is done one can simply extend $\Gamma$ arbitrarily to a $C^1$ horizontal curve on $[0,1]$, since $\mathcal{L}^{1}([0,1]\setminus [m,M])<\varepsilon$. 

In the following lemma we study the curve $\gamma$ near points of $K$. We pay particular attention to movement in non horizontal directions; using group translations we bring the point of interest to $0$, where the horizontal directions in $\mathbb{R}^{2n+1}$ correspond to vectors of the form $(v,0)$ for $v\in \mathbb{R}^{2n}$.

\begin{lemma}\label{studyofcurve}
There is a constant $C$ for which the following holds. Let $\eta>0$ and $t_{0}\in K$. Denote $\varphi=\gamma(t_{0})^{-1}\gamma$ so that $\varphi(t_{0})=0$. Then for any $t_{0}<t<t_{0}+\delta(\eta)$:
\begin{itemize}
\item $|\gamma(t)-\gamma(t_{0}) - (t-t_{0})\gamma'(t_{0})| \leq \eta (t-t_{0})$.
\item $|\varphi(t)-(t-t_{0})(v,0)| \leq C \eta (t-t_{0})$ where $v=(\gamma_{1}'(t_{0}),\ldots, \gamma_{2n}'(t_{0}))$.
\item $|\varphi_{2n+1}(t)|\leq C\eta(|v|+\eta)(t-t_{0})^{2}$.
\end{itemize}
\end{lemma}

\begin{proof}
From the definition of $\delta(\eta)$ when we chose the set $K$, we deduce
\[|\gamma(t)-\gamma(t_{0}) - (t-t_{0})\gamma'(t_{0})| \leq \int_{t_{0}}^{t} |\gamma'-\gamma'(t_{0})| \leq \eta (t-t_{0})\]
for $t_{0}<t<t_{0}+\delta(\eta)$. 

The group translation is Euclidean in the first $2n$ coordinates and its differential maps horizontal vectors to horizontal vectors. Since horizontal vectors at $0$ are of the form $(v,0)$ for $v\in \mathbb{R}^{2n}$, setting $v=(\gamma_{1}'(t_{0}),\ldots, \gamma_{2n}'(t_{0}))$ gives $\varphi'(t_{0})=(v,0)$. We next claim:
\begin{equation}\label{tocheck}
\gamma(t_{0})^{-1}(\gamma(t_{0})+(t-t_{0})\gamma'(t_{0}))=(t-t_{0})(v,0).
\end{equation}
It is clear that the first $2n$ coordinates on each side of \eqref{tocheck} coincide, since the group operation is Euclidean in the first $2n$ coordinates. An elementary calculation using the formula for the group translation shows that the final coordinate of the left hand side of \eqref{tocheck} is given by:
\[(t-t_{0})(\gamma_{2n+1}'(t_{0})-2\sum_{i=1}^{n}(\gamma_{n+i}(t_{0})\gamma_{i}'(t_{0})-\gamma_{i}(t_{0})\gamma_{n+i}'(t_{0}))).\]
The assumption $t_{0}\in K$ implies, by definition, that $\gamma'(t_{0})$ is a horizontal vector at $\gamma(t_{0})$. Hence, as in Lemma \ref{lift},
\[\gamma_{2n+1}'(t_{0})=2\sum_{i=1}^{n}(\gamma_{n+i}(t_{0})\gamma_{i}'(t_{0})-\gamma_{i}(t_{0})\gamma_{n+i}'(t_{0})).\]
This implies that the final coordinate of the left hand side of \eqref{tocheck} is equal to zero, completing the proof of \eqref{tocheck}

Since left translations are diffeomorphisms, they change Euclidean distances on compact sets by at most a constant factor. This implies
\[ |\varphi(t)-(t-t_{0})(v,0)| \leq C |\gamma(t)-\gamma(t_{0}) - (t-t_{0})\gamma'(t_{0})| \leq C \eta (t-t_{0})\]
for all $t_{0}<t<t_{0}+\delta(\eta)$, where $C$ is a constant depending only on a compact set containing the image of $\gamma$.

Fix $t_{0}<t<t_{0}+\delta(\eta)$. By using Lemma \ref{lift} and observing $\varphi(t_{0})=0$, we know that
\[\varphi_{2n+1}(t)=2\sum_{i=1}^{n}\int_{t_{0}}^{t} (\varphi_{i}'\varphi_{n+i}-\varphi_{n+i}'\varphi_{i}).\]
To bound $|\varphi_{2n+1}(t)|$ we estimate each integral in the sum individually, using exactly the same argument. Fix $1\leq i\leq n$; we estimate
\begin{align*}
& \int_{t_{0}}^{t} |\varphi_{i}'(s)\varphi_{n+i}(s)-(s-t_{0})v_{i}v_{n+i}| \dd s\\
&\qquad \leq \int_{t_{0}}^{t} |\varphi_{i}'(s)| |\varphi_{n+i}(s)-(s-t_{0})v_{n+i}| + (s-t_{0})|v_{n+i}||\varphi_{i}'(s)-v_{i}| \dd s.
\end{align*}
For the first term we use the inequality $|\varphi_{n+i}(s)-(s-t_{0})v_{n+i}|\leq C\eta(s-t_{0})$, which was proved above, to obtain
\begin{align*}
\int_{t_{0}}^{t} |\varphi_{i}'(s)| |\varphi_{n+i}(s)-(s-t_{0})v_{n+i}|\dd s &\leq C\eta \int_{t_{0}}^{t} |\varphi_{i}'(s)| (s-t_{0})\dd s\\
&\leq C\eta(t-t_{0})\int_{t_{0}}^{t} (|\varphi_{i}'(s)-v_{i}|+|v_{i}|)\dd s\\
&\leq C\eta (\eta+|v_{i}|)(t-t_{0})^{2}.
\end{align*}
For the second term we have:
\begin{align*}
\int_{t_{0}}^{t} (s-t_{0})|v_{n+i}||\varphi_{i}'(s)-v_{i}| \dd s &\leq |v_{n+i}|(t-t_{0})\int_{t_{0}}^{t} |\varphi_{i}'(s)-v_{i}| \dd s \\
&\leq \eta |v_{n+i}|(t-t_{0})^{2}.
\end{align*}
Hence we obtain
\[\int_{t_{0}}^{t} |\varphi_{i}'(s)\varphi_{n+i}(s)-(s-t_{0})v_{i}v_{n+i}| \dd s\leq C\eta(|v_{i}|+|v_{n+i}|+\eta)(t-t_{0})^{2}.\]
In exactly the same way we find
\[ \int_{t_{0}}^{t} |\varphi_{n+i}'(s)\varphi_{i}(s)-(s-t_{0})v_{i}v_{n+i}| \dd s \leq C\eta(|v_{i}|+|v_{n+i}|+\eta)(t-t_{0})^{2}.\]
Using the triangle inequality, we obtain
\[ \int_{t_{0}}^{t} (\varphi_{i}'\varphi_{n+i}-\varphi_{n+i}'\varphi_{i}) \leq C\eta(|v_{i}|+|v_{n+i}|+\eta)(t-t_{0})^{2}.\]
By summing these estimates over $1\leq i\leq n$, we obtain
\[|\varphi_{2n+1}(t)|\leq C\eta(|v|+\eta)(t-t_{0})^{2}.\]
as required.
\end{proof}

Since $m, M\in K$, the connected components of $[m,M]\setminus K$ are open intervals, which we denote by $(a_{i},b_{i})$ for $i\geq 1$. Thus $[m,M]\setminus K=\bigcup_{i\geq 1}(a_{i},b_{i})$, with $a_{i}, b_{i}\in K$. Since $(b_{i}-a_{i})\to 0$, we may use the definition of $K$ and Lemma \ref{studyofcurve} to find $\varepsilon_{i} \downarrow 0$ such that for every $i\geq 1$:
\begin{align*}
|\gamma'(b_{i})-\gamma'(a_{i})|&\leq \varepsilon_{i},\\
|\gamma(b_{i})-\gamma(a_{i})-(b_{i}-a_{i})\gamma'(a_{i})|&\leq \varepsilon_{i}(b_{i}-a_{i}),\\
|\gamma(a_{i})^{-1}\gamma(b_{i})-(b_{i}-a_{i})(\gamma_{1}'(a_{i}), \ldots, \gamma_{2n}'(a_{i}),0)| &\leq \varepsilon_{i} (b_{i}-a_{i}),
\end{align*}
and
\[|(\gamma(a_{i})^{-1}\gamma(b_{i}))_{2n+1}| \leq \varepsilon_{i} (\varepsilon_{i}+|\gamma_{1}'(a_{i})| + \ldots + |\gamma_{2n}'(a_{i})|)(b_{i}-a_{i})^{2}.\]
By replacing $\varepsilon_{i}$ by $\max \{\varepsilon_{i}, (b_{i}-a_{i})\}$ we may assume $(b_{i}-a_{i})\leq \varepsilon_{i}$ for $i\geq 1$.

We will construct $C^{1}$ interpolations in the gaps $(a_{i},b_{i})\subset [m,M]\setminus K$. Inspired by Lemma \ref{lift} and Lemma \ref{area} we first state a preliminary lemma for constructing curves in $\mathbb{R}^2$. It asserts the existence of a curve in $\mathbb{R}^{2}$ with controlled derivative which satisfies boundary conditions and sweeps out a given signed area.

\begin{lemma}\label{plane}
Let $a<b$ and $\eta>0$. Suppose $p,v,w\in \mathbb{R}^2$ and $A\in \mathbb{R}$ satisfy the following estimates:
\begin{itemize}
\item $|v-w|\leq \eta$.
\item $|p-(b-a)v|\leq \eta(b-a)$.
\item $|A|\leq \eta (\eta + |v|)(b-a)^2$.
\end{itemize}
Then there is a $C^{1}$ curve $\sigma \colon [a,b] \to \mathbb{R}^2$ such that:
\begin{itemize}
\item $\sigma(a)=0$ and $\sigma(b)=p$.
\item $\sigma'(a)=v$ and $\sigma'(b)=w$.
\item The signed area bounded by the line segment $[0,p]$ and $\sigma$ is $A$.
\item $|\sigma'(t)-v|\leq C\eta$ for all $t\in [a,b]$, for some constant $C$.
\end{itemize}
\end{lemma}

\begin{proof}
Fix arbitrary $T>0$ and $P, V, W\in \mathbb{R}^{2}$. Consider first the auxiliary curve $f\colon [0,T]\to \mathbb{R}^{2}$ given by:
\begin{equation}\label{aux}
f(t)=tV+\frac{t^2}{T^2}(3P-2TV-TW)+\frac{t^3}{T^3}(-2P+TV+TW).
\end{equation}
Direct computation shows $f(0)=0$, $f(T)=P$, $f'(0)=V$ and $f'(T)=W$. If $|V-W|\leq \delta$ and $|P-TV|\leq \delta T$ for some $\delta>0$, it is easy to check that $|f'(t)-V|\leq 17\delta$ for $0\leq t\leq T$.

To construct the curve in the statement of the lemma, we consider separately the cases $|v|\leq 2\eta$ and $|v|>2\eta$. 

\medskip

\noindent \emph{Suppose $|v|\leq 2\eta$. }

In this case $|w|\leq 3\eta$ and $|p|\leq 3\eta (b-a)$. We claim there exists a $C^1$ curve $\alpha \colon [a,(a+b)/2]\to \mathbb{R}^2$ satisfying:
\begin{itemize}
\item $\alpha(a)=0$ and $\alpha((a+b)/2)=p$.
\item $\alpha'(a)=v$ and $\alpha'((a+b)/2)=w$.
\item $|\alpha'| \leq C\eta$ for some constant $C$. 
\end{itemize}
Indeed, let $\widetilde{\alpha}\colon [0,(b-a)/2]\to \mathbb{R}^{2}$ be the curve in \eqref{aux} with $T=(b-a)/2$, $P=p$, $V=v$, $W=w$ and $\delta=8\eta$. Then $|V-W|\leq \eta$ and
\[|P-TV|\leq |P|+T|V|\leq 4\eta (b-a).\]
Hence $|\widetilde{\alpha}'(t)-V|\leq 136\eta$ for $0\leq t\leq (b-a)/2$. The desired curve $\alpha$ is then given by $\alpha(t)=\widetilde{\alpha}(t-a)$.

Let $A_{\alpha}$ denote the signed area of the region bounded by $[0,p]$ and $\alpha$. Since $|\alpha'|\leq C\eta$, we know $|\alpha| \leq C\eta(b-a)$ and hence:
\[|A_{\alpha}| = \frac{1}{2} \left| \int_{a}^{(a+b)/2} (\alpha_{1}\alpha_{2}'-\alpha_{2}\alpha_{1}') \right| \leq C^2 \eta^2(b-a)^2.\]
The assumption $|v|\leq 2\eta$ gives $|A|\leq 3\eta^{2}(b-a)^{2}$ and therefore:
\[|A-A_{\alpha}| \leq |A|+|A_{\alpha}| \leq (3+C^2)\eta^2 (b-a)^2.\]

We claim that there exists a $C^1$ curve $\beta\colon [(a+b)/2,b]\to \mathbb{R}^2$ satisfying:
\begin{itemize}
\item $\beta((a+b)/2)=\beta(b)=p$.
\item $\beta'((a+b)/2)=\beta'(b)=w$.
\item The signed area enclosed by $\beta$ is $A-A_{\alpha}$.
\item $|\beta'|\leq C\eta$ for some new, potentially larger, constant $C$.
\end{itemize}
Write $T=(b-a)/2$. By using translation, rotation and reparameterization, it suffices to give a $C^1$ curve $g\colon [0,T]\to \mathbb{R}^{2}$ such that:
\begin{itemize}
\item $g(0)=g(T)=0$.
\item $g'(0)=g'(T)=(0,|w|)$.
\item The signed area enclosed by $g$ is $A-A_{\alpha}$.
\item $|g'|\leq C\eta$.
\end{itemize}
To achieve this we will follow a circle at varying speed. Let $r=\sqrt{|A-A_{\alpha}|/\pi}$. Since $|w|\leq 3\eta$, we may fix a $C^1$ map $\theta\colon [0,T]\to 2\pi$ such that:
\begin{itemize}
\item $\theta(0)=0$ and $\theta(T)=2\pi$.
\item $\theta'(0)=\theta'(T)=|w|/r$.
\item $|\theta'|\leq C\max \{\eta/r, 2\pi / T\}$.
\end{itemize}
Define the curve $g\colon [0,T]\to \mathbb{R}^{2}$ by
\[g(t)=(-r+r\cos \theta(t),r\sin \theta(t)).\]
Clearly $g(0)=g(T)=0$ and
\[g'(t)=\theta'(t)(-r\sin \theta(t), r\cos \theta(t)).\]
Hence $g'(0)=g'(T)=(0,|w|)$ and
\[|g'|\leq r|\theta'|\leq C\max \{\eta, 2\pi r / T\}.\]
Our earlier estimate of $|A-A_{\alpha}|$ implies $|A-A_{\alpha}|\leq C\eta^2 T^2$. Using the definition of $r$ gives $r/T\leq C\eta$ and hence $|g'|\leq C\eta$ as desired. Since $\theta(0)=0$ and $\theta(T)=2\pi$, the signed area enclosed by $g$ is $\pi r^{2}=|A-A_{\alpha}|$. This completes the justification of the existence of $g$ if the signed area is positive. To instead obtain a curve with negative signed area, use a circle with centre $(r,0)$ parameterized in a clockwise direction. Our construction of the curve $g$ implies the existence of a satisfactory curve $\beta$.

To finish the proof of the case $|v|\leq 2\eta$, we define $\sigma$ as the concatentation of the curves $\alpha$ and $\beta$. Clearly $\sigma$ has the properties given in the statement of the lemma.

\medskip

\noindent \emph{Suppose $|v|>2\eta$.}

We first claim there is a $C^1$ curve $\alpha\colon [a,(2a+b)/3]\to \mathbb{R}^2$ satisfying:
\begin{itemize}
\item $\alpha(a)=0$ and $\alpha((2a+b)/3)=p/3$.
\item $\alpha'(a)=v$ and $\alpha'((2a+b)/3)=p/(b-a)$.
\item $|\alpha'-v|\leq C\eta$.
\end{itemize}
To see this, let $\widetilde{\alpha}\colon [0,(b-a)/3]\to \mathbb{R}^{2}$ be the curve given by \eqref{aux} with $T=(b-a)/3$, $P=p/3$, $V=v$, $W=p/(b-a)$ and $\delta=\eta$. Clearly then
\[|V-W|= |v-p/(b-a)| \leq \eta\]
and
\[|P-TV|=|p-(b-a)v|/3\leq \eta(b-a)/3.\]
Hence $|\widetilde{\alpha}'(t)-v|\leq C\eta$ for $0\leq t\leq (b-a)/3$. The desired curve $\alpha$ is given by $\alpha(t)=\widetilde{\alpha}(t-a)$.

Similarly, since $|w-v|\leq \eta$, we can fix a $C^1$ curve $\zeta\colon [(a+2b)/3,b]\to \mathbb{R}^2$ satisfying:
\begin{itemize}
\item $\zeta((a+2b)/3)=2p/3$ and $\zeta(b)=p$.
\item $\zeta'((a+2b)/3)=p/(b-a)$ and $\zeta'(b)=w$.
\item $|\zeta'-v|\leq C\eta$.
\end{itemize}

The assumption $|v|>2\eta$, together with $|p-(b-a)v|\leq \eta(b-a)$, implies $|p|\geq \eta(b-a)$ and $(b-a)|v|\leq 2|p|$. These then imply
\[|A|\leq \eta (\eta + |v|)(b-a)^2 \leq 3\eta(b-a)|p|.\]
Let $A_{\alpha}$ be the signed area enclosed by $[0,p/3]$ and $\alpha$. Similarly let $A_{\zeta}$ be the signed area enclosed by $[2p/3,p]$ and $\zeta$. We claim $|A_{\alpha}|, |A_{\zeta}| \leq C\eta(b-a)|p|$. This is most intuitive from the geometry, but we now check formally. For the first we use Lemma \ref{area} to estimate as follows:
\begin{align*} 
|A_{\alpha}| &= \frac{1}{2} \left| \int_{a}^{(2a+b)/3} (\alpha_{1}\alpha_{2}' - \alpha_{2}\alpha_{1}') \right|\\
&\leq \frac{1}{2}\int_{a}^{(2a+b)/3} (|\alpha_{1}(s)\alpha_{2}'(s) - sv_{1}v_{2}| + |\alpha_{2}(s)\alpha_{1}'(s)-sv_{1}v_{2}|) \dd s.
\end{align*}
For the first term have
\[|\alpha_{1}(s)\alpha_{2}'(s) - sv_{1}v_{2}|\leq |\alpha_{1}(s)||\alpha_{2}'(s)-v_{2}| + |v_{2}||\alpha_{1}(s)-sv_{1}|.\]
The inequality $|\alpha'-v|\leq C\eta$ implies $|\alpha(s)-sv|\leq C\eta s$ and hence, since we assumed $\eta<|v|/2$, $|\alpha(s)|\leq C|v|s$. Thus
\begin{align*}
|\alpha_{1}(s)\alpha_{2}'(s) - sv_{1}v_{2}| &\leq C|v|s\cdot C\eta + |v| \cdot C\eta s\\
&\leq C|v|\eta s.
\end{align*}
Hence, using also the earlier inequality $(b-a)|v|\leq 2|p|$,
\begin{align*}
\int_{a}^{(2a+b)/3} |\alpha_{1}(s)\alpha_{2}'(s) - sv_{1}v_{2}| \dd s &\leq \int_{a}^{(2a+b)/3} C|v|\eta s \dd s\\
&\leq C|v|\eta (b-a)^2\\
&\leq C\eta(b-a)|p|.
\end{align*}
Almost the same argument gives
\[\int_{a}^{(2a+b)/3} |\alpha_{2}(s)\alpha_{1}'(s) - sv_{1}v_{2}| \dd s \leq C\eta(b-a)|p|\]
and hence $|A_{\alpha}|\leq C\eta(b-a)|p|$. A similar argument gives $|A_{\zeta}|\leq C\eta(b-a)|p|$, as claimed.

Putting together our estimates, we deduce:
\[|A-A_{\alpha}-A_{\zeta}|\leq |A|+|A_{\alpha}|+|A_{\zeta}| \leq C\eta(b-a)|p|.\]
We now claim that there exists a $C^1$ curve $\beta\colon [(2a+b)/3, (a+2b)/3]\to \mathbb{R}^2$ satisfying:
\begin{itemize}
\item $\beta((2a+b)/3)=p/3$ and $\beta((a+2b)/3)=2p/3$.
\item $\beta'((2a+b)/3)=\beta'((a+2b)/3)=p/(b-a)$.
\item The signed area enclosed by $[p/3, 2p/3]$ and $\beta$ is $A-A_{\alpha}-A_{\zeta}$.
\item $|\beta'-p/(b-a)|\leq C\eta$.
\end{itemize}
Using translation, rotation and reparameterization, it suffices to define a $C^1$ curve $h\colon [0,(b-a)/3]\to \mathbb{R}^{2}$ such that:
\begin{itemize}
\item $h(0)=0$ and $h((b-a)/3)=(|p|/3,0)$.
\item $h'(0)=h'((b-a)/3)=(|p|/(b-a),0)$.
\item The signed area enclosed by $[0,(|p|/3,0)]$ and $h$ is $A-A_{\alpha}-A_{\zeta}$.
\item $|h'- (|p|/(b-a),0) |\leq C\eta$.
\end{itemize}
Denote $T=(b-a)/3$ and fix $\lambda \in \mathbb{R}$. Define a $C^1$ curve $h\colon [0,T]\to \mathbb{R}^2$ by
\[h(t)=((|p|/3T)t, \lambda \sin^{2} (\pi t/T))\]
so that
\[h'(t)=( |p|/3T, (\pi \lambda / T)\sin(2\pi t/T)).\]
It is easy to verify that
\[h(0)=0,\qquad h((b-a)/3)=(|p|/3,0)\]
and
\[h'(0)=h'((b-a)/3)=(|p|/(b-a),0).\]
It is also clear that
\[|h'- (|p|/(b-a),0)| = |h'-(|p|/3T,0)|\leq \pi|\lambda|/T.\]
A more lengthy, but still elementary, calculation shows
\[\int_{0}^{T} (h_{1}h_{2}'-h_{2}h_{1}')=-\lambda |p|/3.\]
Hence the signed area of the region enclosed by $[0,(|p|/3,0)]$ and $h$ is $-\lambda |p|/6$. Choose
\[\lambda = -6(A-A_{\alpha}-A_{\zeta})/|p|\]
so that $-\lambda |p|/6 = A-A_{\alpha}-A_{\zeta}$. Our earlier estimate of $A-A_{\alpha}-A_{\zeta}$ implies $|\lambda|\leq C\eta T$. Hence, with our choice of $\lambda$, we have $|h'-(|p|/3T,0)|\leq C\eta$. Hence $h$ has all the properties required. The existence of the desired curve $\beta$ follows.

We define $\sigma$ as the concatention of the curves $\alpha$, $\beta$ and $\zeta$. Clearly $\sigma$ has the properties given in the statement of the lemma. This concludes the case $|v|>2\eta$ and hence proves the lemma.
\end{proof}

We now use the previous lemma, together with our understanding of horizontal curves, to interpolate smoothly in the gaps $(a_{i},b_{i})$ of $[m,M]\setminus K$.

\begin{lemma}\label{interpolate}
For $i\geq 1$, there is a $C^{1}$ horizontal curve $\psi^{i}\colon [a_{i},b_{i}]\to \mathbb{R}^{2n+1}$ satisfying:
\begin{enumerate}
\item $\psi^{i}(a_{i})=\gamma(a_{i})$ and $\psi^{i}(b_{i})=\gamma(b_{i})$.
\item $(\psi^{i})'(a_{i})=\gamma'(a_{i})$ and $(\psi^{i})'(b_{i})=\gamma'(b_{i})$.
\item $|(\psi^{i})'(t)-\gamma'(a_{i})|\leq C\varepsilon_{i}$ for all $t\in (a_{i},b_{i})$.
\end{enumerate}
\end{lemma}

\begin{proof}
Fix $i\geq 1$ as in the statement of the lemma. Since group translations are diffeomorphisms and preserve horizontal curves, we may assume that $\gamma(a_{i})=0$ for convenience. 

Choose $1\leq J\leq n$ such that the sums
\[|\gamma_{j}'(a_{i})|+|\gamma_{n+j}'(a_{i})|, \qquad 1\leq j\leq n,\]
are maximised by setting $j=J$. We plan to first apply Lemma \ref{plane} with the following parameters:
\begin{itemize}
\item $a=a_{i}, b=b_{i}$, $\eta=2n\varepsilon_{i}$.
\item $p=(\gamma_{J}(b_{i}), \gamma_{n+J}(b_{i}))$.
\item $v=(\gamma_{J}'(a_{i}),\gamma_{n+J}'(a_{i}))$, $w=(\gamma_{J}'(b_{i}),\gamma_{n+J}'(b_{i}))$.
\item $A=-(1/4)\gamma_{2n+1}(b_{i})$.
\end{itemize}
We show that these are admissible parameters for Lemma \ref{plane}. Indeed,
\[|v-w|\leq |\gamma'(b_{i})-\gamma'(a_{i})|\leq \varepsilon_{i} \leq \eta,\]
\[|p-(b-a)v|\leq |\gamma(b_{i})-\gamma(a_{i})-(b_{i}-a_{i})\gamma'(a_{i})| \leq \varepsilon_{i}\leq \eta,\]
and:
\begin{align*}
|A|&\leq |\gamma_{2n+1}(b_{i})|\\
&\leq \varepsilon_{i} (\varepsilon_{i}+|\gamma_{1}'(a_{i})| + \ldots + |\gamma_{2n}'(a_{i})|)(b_{i}-a_{i})^{2}\\
&\leq \varepsilon_{i}(\varepsilon_{i}+n (|\gamma_{J}'(a_{i})|+|\gamma_{n+J}'(a_{i})|))(b_{i}-a_{i})^{2}\\
&\leq \varepsilon_{i}(\varepsilon_{i}+2n |v|)(b_{i}-a_{i})^{2}\\
&\leq \eta(\eta+|v|)(b-a)^2.
\end{align*}
Denote the curve arising from Lemma \ref{plane} with this choice of parameters by $(\psi^{i}_{J}, \psi^{i}_{n+J})\colon [a_{i},b_{i}]\to \mathbb{R}^{2}$.

For each $1\leq j\leq n$ with $j\neq J$ we will again apply Lemma \ref{plane} with the modified parameters:
\begin{itemize}
\item $a=a_{i}, b=b_{i}$, $\eta=\varepsilon_{i}$.
\item $p=(\gamma_{j}(b_{i}), \gamma_{n+j}(b_{i}))$.
\item $v=(\gamma_{j}'(a_{i}),\gamma_{n+j}'(a_{i}))$, $w=(\gamma_{j}'(b_{i}),\gamma_{n+j}'(b_{i}))$.
\item $A=0$.
\end{itemize}
By simplifying the previous argument we see that such parameters are admissible. Denote the curves resulting from applying Lemma \ref{plane} with these parameters by $(\psi^{i}_{j}, \psi^{i}_{n+j})\colon [a_{i},b_{i}]\to \mathbb{R}^{2}$, for $1\leq j\leq n$ and $j\neq J$.

Define a new curve $\psi^{i}\colon [a_{i},b_{i}]\to \mathbb{R}^{2n+1}$ by $\psi^{i}=(\psi_{1}^{i},\ldots, \psi_{2n}^{i},\psi_{2n+1}^{i})$, where
\[\psi_{2n+1}^{i}(t)=2\sum_{j=1}^{n}\int_{a_{i}}^{t} ((\psi_{j}^{i})'\psi_{n+j}^{i}-(\psi_{n+j}^{i})'\psi_{j}^{i}).\]

Clearly $\psi^{i}$ is a $C^{1}$ curve and is horizontal by Lemma \ref{lift}. It follows directly from our application of Lemma \ref{plane} that for $1\leq j\leq 2n$:
\begin{itemize}
\item $\psi_{j}^{i}(a_{i})=\gamma_{j}(a_{i})=0$ and $\psi_{j}^{i}(b_{i})=\gamma_{j}(b_{i})$.
\item $(\psi_{j}^{i})'(a_{i})=\gamma_{j}'(a_{i})$ and $(\psi_{j}^{i})'(b_{i})=\gamma_{j}'(b_{i})$.
\end{itemize}
We check the positions and derivatives for the final coordinate $\psi_{2n+1}^{i}$. Clearly $\psi_{2n+1}^{i}(a_{i})=\gamma_{2n+1}(a_{i})=0$. The equality $(\psi^{i}_{2n+1})'(a_{i})=\gamma_{2n+1}'(a_{i})$ follows by combining the agreement of the initial coordinates:
\begin{align*}
(\psi_{2n+1}^{i})'(a_{i})&=2\sum_{j=1}^{n}((\psi_{j}^{i})'(a_{i})\psi_{n+j}^{i}(a_{i})-(\psi_{n+j}^{i})'(a_{i})\psi_{j}^{i}(a_{i}))\\
&= 2\sum_{j=1}^{n}(\gamma_{j}'(a_{i})\gamma_{n+j}(a_{i})-\gamma_{n+j}'(a_{i})\gamma_{j}(a_{i}))\\
&=\gamma_{2n+1}'(a_{i}).
\end{align*}
Exactly the same argument yields $(\psi_{2n+1}^{i})'(b_{i})=\gamma_{2n+1}'(b_{i})$. From the normalization $\gamma(a_{i})=0$, Lemma \ref{area}, and our application of Lemma \ref{plane}, we have:
\[ 2\int_{a_{i}}^{t} ((\psi_{J}^{i})'\psi_{n+J}^{i}-(\psi_{n+J}^{i})'\psi_{J}^{i})=\gamma_{2n+1}(b_{i})\]
and
\[ 2\int_{a_{i}}^{t} ((\psi_{j}^{i})'\psi_{n+j}^{i}-(\psi_{n+j}^{i})'\psi_{j}^{i})=0, \qquad j\neq J.\]
Summing these equalities gives $\psi_{2n+1}^{i}(b_{i})=\gamma_{2n+1}(b_{i})$.

From our application of Lemma \ref{plane} we know:
\begin{equation}\label{constantderivative}
|(\psi_{j}^{i})'(t)-\gamma_{j}'(a_{i})|\leq C\varepsilon_{i} \mbox{ for }t\in (a_{i},b_{i})\mbox{ and }j=1,\ldots,2n.
\end{equation}
We claim that a similar inequality holds for $j=2n+1$. Since $\gamma'|_{K}$ is continuous, hence bounded, \eqref{constantderivative} implies $(\psi_{j}^{i})'$ is bounded by a constant independent of $i$ and $j$. Using $\psi_{j}^{i}(a_{i})=0$ and $(b_{i}-a_{i})\leq \varepsilon_{i}$, we obtain $|\psi_{j}^{i}(t)|\leq C(b_{i}-a_{i})\leq C\varepsilon_{i}$ for $t\in (a_{i},b_{i})$ and $j=1, \ldots, 2n$. Hence, again using \eqref{constantderivative}, we obtain for every $j=1,\ldots,n,$ and $t\in (a_{i},b_{i})$:
\begin{align*}
|(\psi_{j}^{i})'(t)\psi_{n+j}^{i}(t)-(\psi_{n+j}^{i})'(t)\psi_{j}^{i}(t)| &\leq |(\psi_{j}^{i})'(t)| |\psi_{n+j}^{i}(t)|+|(\psi_{n+j}^{i})'(t) | |\psi_{j}^{i}(t)|\\
&\leq C\varepsilon_{i}.
\end{align*}
Consequently
\[|(\psi_{2n+1}^{i})'(t)| \leq 2\sum_{j=1}^{n}| (\psi_{j}^{i})'(t)\psi_{n+j}^{i}(t)-(\psi_{n+j}^{i})'(t)\psi_{j}^{i}(t)|\leq C\varepsilon_{i}\]
for every $t\in (a_{i},b_{i})$. Since $\gamma(a_{i})=0$ implies $\gamma_{2n+1}'(a_{i})=0$, this completes the proof.
\end{proof}

Using Lemma \ref{interpolate} to interpolate in the subintervals of $[m,M]\setminus K$, we define a curve $\Gamma\colon [m,M]\to \mathbb{R}^{2n+1}$ by:
\[ \Gamma(t) = \begin{cases} \gamma(t) & \mbox{if }t\in K,\\
\psi^{i}(t) & \mbox{if }t\in (a_{i},b_{i}) \mbox{ and } i\geq 1. \end{cases}\]

We now show that $\Gamma$ is the desired $C^{1}$ horizontal curve.

\begin{proposition}\label{smoothcurve}
The function $\Gamma\colon [m,M]\to \mathbb{R}^{2n+1}$ defines a $C^{1}$ horizontal curve in the Heisenberg group $\mathbb{H}^{n}$ with derivative given by:
\[ \Gamma'(t) = \begin{cases} \gamma'(t) & \mbox{if }t\in K,\\
(\psi^{i})'(t) & \mbox{if }t\in (a_{i},b_{i}) \mbox{ and } i\geq 1. \end{cases}\]
\end{proposition}

\begin{proof}
Clearly $\Gamma$ is $C^{1}$ on each interval $(a_{i},b_{i})$ with $\Gamma'=(\psi^{i})'$. We check that $\Gamma'(t)$ exists and is equal to $\gamma'(t)$ for $t\in K$.

We first check differentiability from the right. If $t=a_{i}$ for some $i$ then differentiability on the right of $\Gamma$ at $t$ follows from that of $\psi_{i}$ and the agreements $\psi^{i}(a_{i})=\gamma(a_{i})$ and $(\psi^{i})'(a_{i})=\gamma'(a_{i})$. Hence we suppose $t\neq a_{i}$ for any $i$. Let $\delta>0$. If $t+\delta\in K$ then
\[|\Gamma(t+\delta)-\Gamma(t)-\delta\gamma'(t)|=|\gamma(t+\delta)-\gamma(t)-\delta \gamma'(t)|\]
which is small relative to $\delta$, for $\delta$ sufficiently small, by differentiability of $\gamma$ at $t$. Hence we suppose $t+\delta\notin K$ which implies $t< a_{i}<t+\delta<b_{i}$ for some $i\geq 1$; by choosing $\delta$ small we can ensure $i$ is large and hence $\varepsilon_{i}$ is small. In this case we estimate as follows:
\begin{align*}
|\Gamma(t+\delta)-\Gamma(t)-\delta\gamma'(t)|&\leq |\Gamma(t+\delta)-\Gamma(a_{i})-(t+\delta-a_{i}) \gamma'(t)|\\
&\qquad + |\Gamma(a_{i})-\Gamma(t)-(a_{i}-t)\gamma'(t)|\\
&\leq |\psi^{i}(t+\delta)-\psi^{i}(a_{i})-(t+\delta-a_{i})\gamma'(a_{i})|\\
&\qquad + |t+\delta-a_{i}||\gamma'(a_{i})-\gamma'(t)|\\
&\qquad \qquad + |\gamma(a_{i})-\gamma(t)-(a_{i}-t)\gamma'(t)|.
\end{align*}
We observe that the second term is small, relative to $\delta$, by using continuity of $\gamma'|_{K}$ and the simple estimate $|t+\delta-a_{i}|\leq t+\delta-t=\delta$. The third term is small, relative to $\delta$, by differentiability of $\gamma$ at $t$. For the first term we estimate further, using Lemma \ref{interpolate},
\begin{align*}
|\psi^{i}(t+\delta)-\psi^{i}(a_{i})-(t+\delta-a_{i})\gamma'(a_{i})| &= \left| \int_{a_{i}}^{t+\delta} ((\psi^{i})'(s)-\gamma'(a_{i}))\dd s\right|\\
&\leq \int_{a_{i}}^{t+\delta} C\varepsilon_{i} \dd s\\
&\leq C \varepsilon_{i} \delta.
\end{align*}
As indicated above, by choosing $\delta>0$ sufficiently small we can ensure $\varepsilon_{i}$ is also small. This gives the estimate of the final term. Hence $\Gamma$ is differentiable at $t$ from the right with the required derivative $\gamma'(t)$.

Similar arguments, with $t+\delta$ replaced by $t-\delta$, gives differentiability of $\Gamma$ from the left. Hence we deduce that $\Gamma$ is differentiable for $t\in K$ with $\Gamma'(t)=\gamma'(t)$, as claimed.

Clearly the curve $\Gamma$ is everywhere horizontal, since $\gamma$ is horizontal in $K$ and each map $\psi^{i}$ is horizontal in $(a_{i},b_{i})$. We claim that $\Gamma$ is also $C^{1}$. 

Continuity of $\Gamma'$ at points in $[m,M]\setminus K$ is clear since each map $\psi^{i}$ is $C^{1}$. We check continuity of $\Gamma'$ at $t\in K$ from the right. We showed above that $\Gamma'(t)=\gamma'(t)$. If $t=a_{i}$ for some $i$ then continuity of $\Gamma'$ at $t$ follows from continuity of $\psi_{i}'$ and the agreement $\psi_{i}'(a_{i})=\gamma'(a_{i})$. Suppose $t\neq a_{i}$ for any $i$ and let $\delta>0$. If $t+\delta\in K$ then $\Gamma'(t+\delta)=\gamma'(t+\delta)$ which is close to $\gamma'(t)$ for sufficiently small $\delta$ using continuity of $\gamma'|_{K}$. Suppose now that $t+\delta \notin K$. Then we can find $i\geq 1$ such that $t < a_{i}<t+\delta<b_{i}$. As before, if $\delta$ is sufficiently small then we can ensure $\varepsilon_{i}$ is small. In this case we have:
\[|\Gamma'(t+\delta)-\Gamma'(t)|\leq |(\psi^{i})'(t+\delta)-\gamma'(a_{i})|+|\gamma'(a_{i})-\gamma'(t)|\]
which is small for sufficiently small $\delta>0$, using Lemma \ref{interpolate} for the first term and continuity of $\gamma'|_{K}$ for the second. This shows $\Gamma'$ is everywhere continuous on the right; a similar argument gives continuity on the left and concludes the proof.
\end{proof}

Clearly, using Proposition \ref{smoothcurve},
\[\mathcal{L}^{1} (\{t\in [m,M]:\Gamma(t)\neq \gamma(t)\mbox{ or }\Gamma'(t)\neq \gamma'(t)\}) \leq \mathcal{L}^{1}([m,M]\setminus K) <\varepsilon.\]
Proposition \ref{smoothcurve} also states that $\Gamma$ is a $C^{1}$ horizontal curve on $[m,M]$. Extending $\Gamma$ arbitrarily to a $C^{1}$ horizontal curve on $[0,1]$ gives Theorem \ref{lusinheisenberg}.

\section{A Curve Without Approximation in the Engel Group}

We first give the abstract definition of the (first) Engel group, followed by a representation in coordinates \cite{LM10}.

\begin{definition}\label{Engel}
The Engel group $\mathbb{E}$ is the Carnot group whose Lie algebra $\mathcal{E}$ satisfies the decomposition
\[ \mathcal{E}=V_{1}\oplus V_{2}\oplus V_{3},\]
where $V_{1} = \mathrm{span} \{X_{1},X_{2}\}, V_{2}=\mathrm{span}\{X_{3}\}, V_{3}=\mathrm{span}\{X_{4}\}$, and the only non-trivial Lie brackets are
\[ [X_{1}, X_{2}]=X_{3} \text{ and }[X_{1}, X_{3}]=X_{4}.\]
\end{definition}

We represent the Engel group $\mathbb{E}$ by $\mathbb{R}^{4}$ with the vector fields
\begin{equation}\label{engelvectors}
X_{1}=\partial_{1}, \quad X_{2}=\partial_{2}+x_{1}\partial_{3}+\frac{x_{1}^2}{2}\partial_{4},\quad X_{3}=\partial_{3}+x_{1}\partial_{4},\quad X_{4}=\partial_{4}.
\end{equation}
Here $\partial_{i}$ are the coordinate vectors of $\mathbb{R}^4$, which may also be interpreted as operators on differentiable functions. The group law on the Engel group is given by the exponential mapping $\exp\colon \mathcal{E} \to \mathbb{E}$ and the Baker-Campbell-Hausdorff formula. We do not give explicitly the group law but need only note that
\begin{equation}\label{engelmultiplication}
(p_{1},p_{2},p_{3},p_{4})\cdot (0,0,0,t)=(p_{1},p_{2},p_{3},p_{4}+t)
\end{equation}
for any $t\in \mathbb{R}$. This follows directly from the Baker-Campbell-Hausdorff formula.

We can now state our result on the existence of a horizontal curve in the Engel group with no $C^{1}$ horizontal approximation. For emphasis, we recall that a curve $\gamma$ in the Engel group is horizontal if $\gamma'(t)$ is a linear combination of $X_{1}(\gamma(t))$ and $X_{2}(\gamma(t))$ for almost every $t$. In what follows, Lipschitz will mean with respect to the Euclidean metric in both domain and target.

\begin{theorem}\label{nolusin}
Suppose $\varepsilon>0$. Then there exists a Lipschitz horizontal curve $\gamma\colon [0,1]\to \mathbb{R}^{4}$ in the Engel group such that for any $C^{1}$ horizontal curve $\Gamma\colon [0,1]\to \mathbb{R}^{4}$,
\begin{equation}\label{smallagreement}
\mathcal{L}^{1}(\{t\in [0,1]:\gamma(t)=\Gamma(t)\})<\varepsilon.
\end{equation}
\end{theorem}

To prove Theorem \ref{nolusin} we first construct an absolutely continuous horizontal curve $\gamma$ with many small spirals which give movement in the $-\partial_{4}$ direction. We then show, by contradiction, that a $C^{1}$ horizontal curve cannot follow such a path. Given $\varepsilon>0$, choose a sequence $\varepsilon_{i}>0$ and $q_{i}\in \mathbb{Q}\cap (0,1)$ such that:
\begin{itemize}
\item The intervals $(q_{i},q_{i}+\varepsilon_{i})$ are disjoint.
\item The set $S=\bigcup_{i\geq 1} (q_{i},q_{i}+\varepsilon_{i})$ is dense in $(0,1)$.
\item The inequality $\sum_{i} \varepsilon_{i}<\varepsilon$ holds.
\end{itemize}

For each $\delta>0$ we fix a $\delta$-Lipschitz horizontal curve $\Phi_{\delta}\colon [0,\delta] \to \mathbb{R}^{4}$ such that $\Phi_{\delta}(0)=(0,0,0,0)$ and $\Phi_{\delta}(\delta)=\exp(-K\delta^{6} X_{4})=(0,0,0,-K\delta^6)$ for some sufficiently small constant $K>0$. Existence of such curves follows from the Ball-Box Theorem \cite{Gro96}.

We define the horizontal and vertical translations, up to time $t\in [0,1]$, by
\[H(t)=\sum_{q_{i}<t} \varepsilon_{i},\quad V(t)=\sum_{q_{i}<t} K\varepsilon_{i}^6. \]

We define $\gamma \colon [0,1]\to \mathbb{R}^{4}$ by the following formula:
\[ \gamma(t) = \begin{cases} (0,q_{i}-H(q_{i}),0,-V(q_{i}))\cdot \Phi_{\varepsilon_{i}}(t-q_{i}) , & \mbox{if }t\in (q_{i},q_{i}+\varepsilon_{i}), i\geq 1,\\
(0,t-H(t),0,-V(t)) & \mbox{if }t\notin S. \end{cases}\]

Note that \eqref{engelmultiplication} implies $\gamma$ is continuous. Since $\Phi_{\varepsilon_{i}}$ is $\varepsilon_{i}$-Lipschitz and group translations are smooth, it follows that $\gamma$ is $C\varepsilon_{i}$-Lipschitz in each interval $(q_{i},q_{i}+\varepsilon_{i})$, for some constant $C$. 

The important properties of $\gamma$ are given in the following proposition; after we prove the proposition one can forget about the precise formula for $\gamma$. Recall that if $A\subset \mathbb{R}$ is measurable then almost every point $t\in A$ is a Lebesgue density point of $A$, which means that $\mathcal{L}^{1}((t-r,t+r)\cap A)/2r\to 1$ as $r\downarrow 0$.

\begin{proposition}\label{badcurveproperties}
The curve $\gamma\colon [0,1]\to \mathbb{R}^{4}$ is a Lipschitz horizontal curve in the Engel group such that:
\begin{enumerate}
\item $\gamma'=(0,1,0,0)$ almost everywhere in $(0,1)\setminus S$.
\item $\gamma(q_{i}+\varepsilon_{i})=\gamma(q_{i})-(0,0,0,K\varepsilon_{i}^6)$ for every $i\geq 1$.
\end{enumerate}
In particular, the component $\gamma_{4}$ is strictly decreasing in $(0,1)\setminus S$: if $a, b\notin S$ and $a<b$ then $\gamma_{4}(b)<\gamma_{4}(a)$.
\end{proposition}

\begin{proof}
We begin by showing that $\gamma_{1}'(t)=\gamma_{3}'(t)=0$ whenever $t\notin [q_{i}, q_{i}+\varepsilon_{i}]$ for any $i$; the argument is the same for each coordinate. Given such a point $t$ and $\eta>0$, choose $\delta>0$ sufficiently small so that $\varepsilon_{i}<\eta$ whenever $q_{i}\in (t,t+\delta)$. We estimate, since $\gamma$ is $C\varepsilon_{i}$-Lipschitz in $(q_{i},q_{i}+\varepsilon_{i})$,
\begin{align*}
|\gamma_{1}(t+\delta)-\gamma_{1}(t)| &\leq \sum_{t<q_{i}<t+\delta} |\gamma_{1}(\min(q_{i}+\varepsilon_{i},t+\delta))-\gamma_{1}(q_{i})|\\
&\leq \sum_{t<q_{i}<t+\delta} C \varepsilon_{i} ( \min(q_{i}+\varepsilon_{i},t+\delta)-q_{i})\\
&\leq C\eta \sum_{t<q_{i}<t+\delta} ( \min(q_{i}+\varepsilon_{i},t+\delta)-q_{i})\\
&\leq C\eta \delta.
\end{align*}
This shows differentiability on the right and a similar argument shows differentiability on the left. Hence $\gamma_{1}$ is differentiable at $t$ with derivative $0$. Exactly the same argument gives the statement for $\gamma_{3}$.

We now verify that $\gamma_{2}'=1$ almost everywhere in $(0,1)\setminus S$, a similar argument will give the statement for $\gamma_{4}'$. Let $t$ be a Lebesgue density point of $(0,1)\setminus S$, which implies in particular that $t\neq q_{i}$ for any $i$. Given $0<\eta<1$, choose $\delta>0$ sufficiently small so that:
\begin{itemize}
\item $\varepsilon_{i}<\eta$ whenever $t<q_{i}<t+\delta$.
\item $\mathcal{L}^1(S\cap (t,t+2\delta))<\eta \delta$.
\end{itemize}

We first observe that if $t<q_{i}<t+\delta$ then $\varepsilon_{i}<\delta$: if instead $\varepsilon_{i}\geq \delta$ then,
\[(q_{i},q_{i}+\delta) \subset (q_{i},q_{i}+\varepsilon_{i}) \cap (t,t+2\delta) \subset S\cap (t,t+2\delta)\]
which implies
\[\mathcal{L}^1(S\cap (t,t+2\delta))\geq \delta,\]
contradicting the choice of $\delta$.
Hence $(q_{i},q_{i}+\varepsilon_{i})\subset (t,t+2\delta)$ for every $t< q_{i}<t+\delta$. 

We next estimate changes in $H$, using the fact that $t\neq q_{i}$ for any $i$,
\begin{align*}
|H(t+\delta)-H(t)|= \sum_{t<q_{i}<t+\delta} \varepsilon_{i}&=  \sum_{t<q_{i}<t+\delta} \mathcal{L}^1 (q_{i},q_{i}+\varepsilon_{i})\\
&= \mathcal{L}^1 \left( \bigcup_{t<q_{i}<t+\delta} (q_{i},q_{i}+\varepsilon_{i})\right)\\
&\leq \mathcal{L}^1 (S\cap (t,t+2\delta))\\
&<\eta \delta.
\end{align*}

Using the above calculation, and the fact $\gamma$ is $C\varepsilon_{i}$-Lipschitz in $(a_{i},b_{i})$, we now estimate the increment of $\gamma_{2}$:
\begin{align*}
|\gamma_{2}(t+\delta)-\gamma_{2}(t)-\delta| & \leq |H(t+\delta)-H(t)|\\
&\qquad + \sum_{t<q_{i}<t+\delta} |\gamma_{2}(\min(q_{i}+\varepsilon_{i},t+\delta))-\gamma_{2}(q_{i})|\\
&\leq \eta\delta + \sum_{t<q_{i}<t+\delta} C\varepsilon_{i} ( \min(q_{i}+\varepsilon_{i},t+\delta)-q_{i})\\
&\leq \eta\delta + C\eta \sum_{t<q_{i}<t+\delta} ( \min(q_{i}+\varepsilon_{i},t+\delta)-q_{i})\\
&\leq C\eta \delta.
\end{align*}

This shows differentiability of $\gamma_{2}$ on the right and a similar argument shows differentiability of $\gamma_{2}$ on the left. Hence $\gamma_{2}$ is differentiable at $t$ with $\gamma_{2}'=1$. A similar argument, with $H$ replaced by $V$, gives the claim for $\gamma_{4}$. We have now shown (1).

We now know that $\gamma(t)=(0,t-H(t),0,-V(t))$ and $\gamma'(t)=(0,1,0,0)$ at almost every point $t\in [0,1]\setminus S$. Hence, by examining the form of the horizontal vector fields in the Engel group \eqref{engelvectors}, it is clear that $\gamma'$ is a horizontal vector at almost every point of $[0,1]\setminus S$. Recall next that the maps $\Phi_{\varepsilon_{i}}$ are Lipschitz horizontal curves. Since group translations are smooth and send horizontal curves to horizontal curves, it follows that $\gamma$ is a Lipschitz horizontal curve when restricted to any subinterval of $S$. Hence $\gamma'$ exists and is a horizontal vector at almost every point of $[0,1]$.

To prove (2) we simply observe that for each $i\geq 1$:
\[\gamma_{2}(q_{i}+\varepsilon_{i})-\gamma_{2}(q_{i}) = q_{i}+\varepsilon_{i}-H(q_{i}+\varepsilon_{i})-q_{i}+H(q_{i})=0,\]
and
\[\gamma_{4}(q_{i}+\varepsilon_{i})-\gamma_{4}(q_{i}) = -V(q_{i}+\varepsilon_{i})+V(q_{i})=-K\varepsilon_{i}^6.\]

To complete the proof of the proposition, we check that $\gamma$ is absolutely continuous. To do this we verify that increments $\gamma(t_{2})-\gamma(t_{1})$ are given by integrating $\gamma'$ from $t_{1}$ to $t_{2}$. We know that $\Phi_{\varepsilon_{i}}$ and hence $\gamma$ are absolutely continuous curves in the subintervals of $S$. Hence we may assume that $t_{1}, t_{2}\notin [q_{i},q_{i}+\varepsilon_{i}]$ for any $i$. We now calculate, using the absolute continuity of $\gamma$ in the subintervals of $S$,
\begin{align*}
\int_{t_{1}}^{t_{2}} \gamma'(t)\dd t &= \int_{(t_{1},t_{2})\setminus S} (0,1,0,0) \dd t + \sum_{t_{1}<q_{i}<t_{2}} (\gamma(q_{i}+\varepsilon_{i})-\gamma(q_{i}))\\
&= (0,\mathcal{L}^{1}((t_{1},t_{2})\setminus S),0,0) + (0,0,0,\sum_{t_{1}<q_{i}<t_{2}} (-K\varepsilon_{i}^6))
\end{align*}
Our assumption $t_{1}, t_{2}\notin [q_{i},q_{i}+\varepsilon_{i}]$ for any $i$ implies that if $(q_{i},q_{i}+\varepsilon_{i})\subset S$ intersects $(t_{1}, t_{2})$, then $(q_{i},q_{i}+\varepsilon_{i})\subset (t_{1},t_{2})$. Hence we may calculate the second component as follows:
\begin{align*}
\mathcal{L}^{1}((t_{1},t_{2})\setminus S) &=(t_{2}-t_{1})-\mathcal{L}^{1}((t_{1},t_{2})\cap S)\\
&= (t_{2}-t_{1})-\sum_{t_{1}<q_{i}<t_{2}} \mathcal{L}^{1}(q_{i},q_{i}+\varepsilon_{i})\\
&= (t_{2}-t_{1})-\sum_{t_{1}<q_{i}<t_{2}} \varepsilon_{i}\\
&= (t_{2}-t_{1})- (H(t_{2})-H(t_{1}))\\
&= \gamma_{2}(t_{2})-\gamma_{2}(t_{1}).
\end{align*}

For the fourth component we have:
\[ \sum_{t_{1}<q_{i}<t_{2}} (-K\varepsilon_{i}^6)=-(V(t_{2})-V(t_{1}))=\gamma_{4}(t_{2})-\gamma_{4}(t_{1}).\]

Hence
\[\int_{t_{1}}^{t_{2}} \gamma'(t)\dd t= \gamma(t_{2})-\gamma(t_{1}),\]
so $\gamma$ is absolutely continuous. Since $\gamma$ is absolutely continuous with a derivative bounded almost everywhere, we conclude that $\gamma$ is Lipschitz.
\end{proof}

We now proceed by contradiction. Suppose $\Gamma\colon [0,1]\to \mathbb{R}^{4}$ is a $C^{1}$ horizontal curve such that \eqref{smallagreement} fails, namely:
\begin{equation}\label{smallagreementfails}
\mathcal{L}^{1}(\{t\in [0,1]:\Gamma(t)=\gamma(t)\}) \geq \varepsilon.
\end{equation}
Then, since $\mathcal{L}^{1}(S)<\varepsilon$, the set
\[A=\{t\in [0,1] \setminus S: \Gamma(t)=\gamma(t)\}\]
has positive Lebesgue measure. 

Almost every point $t_{0}$ of $A$ is a Lebesgue density point of $A$ and, by locality of the derivative, satisfies $\Gamma_{2}'(t_{0})=\gamma_{2}'(t_{0})=1$. Fix such a point $t_{0} \in A$. Using continuity of $\Gamma_{2}'$ at $t_{0}$, we fix $\delta>0$ sufficiently small so that
\[\Gamma_{2}'(t)\geq 0 \mbox{ for all }t\in (t_{0},t_{0}+\delta)\]
and note that, since $t_{0}$ is a Lebesgue density point of $A$, we necessarily have
\begin{equation}\label{setAisbig}
\mathcal{L}^{1}( A\cap (t_{0},t_{0}+\delta) )>0.
\end{equation}

Using the fact that $\Gamma_{2}$ is increasing in $(t_{0},t_{0}+\delta)$ and $\Gamma$ is a horizontal curve, we will now deduce that $\Gamma_{4}\geq \Gamma(t_{0})$. Notice that the behaviour of $\Gamma_{4}$ is substantially different to that of $\gamma_{4}$ in Proposition \ref{badcurveproperties}.

\begin{proposition}\label{smoothcurveproperties}
For any $t\in (t_{0},t_{0}+\delta)$ it holds that
\[\Gamma_{4}(t)\geq \Gamma_{4}(t_{0}).\]
\end{proposition}

\begin{proof}
Since $\Gamma$ is a $C^1$ horizontal curve, there are $a, b\colon [0,1]\to \mathbb{R}$ such that, for every $t\in[0,1]$:
\begin{align*}
\Gamma'(t) &= a(t)X_{1}(\Gamma(t)) + b(t)X_{2}(\Gamma(t))\\
&= (a(t),b(t),b(t)\Gamma_{1}(t),b(t)\Gamma_{1}(t)^{2}/2).
\end{align*}
Since $\Gamma_{2}'(t)\geq 0$ for all $t\in (t_{0}, t_{0}+\delta)$ it follows $b(t)\geq 0$ and consequently $\Gamma_{4}'(t)\geq 0$ for all $t\in (t_{0},t_{0}+\delta)$. It follows that
\[\Gamma_{4}(t)=\Gamma_{4}(t_{0})+\int_{t_{0}}^{t} \Gamma_{4}'(t) \dd t \geq \Gamma_{4}(t_{0})\]
for any $t\in (t_{0},t_{0}+\delta)$, as required.
\end{proof}

We now observe that the different behaviour of $\Gamma_{4}$ and $\gamma_{4}$ shows the curves cannot coincide. By definition, $t_{0}\in A$ which implies $\Gamma(t_{0})=\gamma(t_{0})$. Further, using \eqref{setAisbig}, it holds that
\[ \mathcal{L}^{1}(\{t\in (t_{0},t_{0}+\delta)\setminus S: \Gamma(t)=\gamma(t)\})>0. \]
However, for any $t\in (t_{0},t_{0}+\delta)\setminus S$ we have:
\begin{itemize}
\item $\gamma_{4}(t)<\gamma_{4}(t_{0})$ by Proposition \ref{badcurveproperties}.
\item $\Gamma_{4}(t)\geq \Gamma_{4}(t_{0})$ by Proposition \ref{smoothcurveproperties}.
\end{itemize}
Hence for any $t\in (t_{0},t_{0}+\delta)\setminus S$ we obtain,
\[ \Gamma_{4}(t)\geq \Gamma_{4}(t_{0})=\gamma_{4}(t_{0})>\gamma_{4}(t)\]
and in particular $\Gamma_{4}(t)\neq \gamma_{4}(t)$. This gives the desired contradiction to the existence of the curve $\Gamma$ satisfying \eqref{smallagreementfails}, yielding Theorem \ref{nolusin}.

\end{document}